\documentclass[12pt]{article}
\usepackage{amssymb, amsfonts, amsmath, bbm}
\usepackage{amsthm}
\usepackage{dsfont}
\usepackage{hyperref}
\usepackage{xcolor}
\hypersetup{
	colorlinks,
	linkcolor={red!60!black},
	citecolor={green!60!black}
}

\usepackage{lipsum}

\newcommand\blfootnote[1]{%
  \begingroup
  \renewcommand\thefootnote{}\footnote{#1}%
  \addtocounter{footnote}{-1}%
  \endgroup
}

\newcommand{\rtp}{{\Bar{\otimes}}}
\newcommand{\rtpd}{{\rtp_\delta}}

\numberwithin{equation}{section}

\begin{document}
\newtheorem{theorem}{Theorem}[section]
\newtheorem{definition}[theorem]{Definition}
\newtheorem{lemma}[theorem]{Lemma}
\newtheorem{note}[theorem]{Note}
\newtheorem{corollary}[theorem]{Corollary}
\newtheorem{proposition}[theorem]{Proposition}
\renewcommand{\theequation}{\arabic{section}.\arabic{equation}}
\newcommand{\newsection}[1]{\setcounter{equation}{0} \section{#1}}
\title{Tensor Products of Ideals and Projection Bands}
\author{Mohamed Amine BEN AMOR\blfootnote{${}^\star$mohamedamine.benamor@ipest.rnu.tn} ${}^\star$\\ \"Omer GOK\blfootnote{${}^\dagger $ gok@yildiz.edu.tr}${}^\dagger $ \\ Damla YAMAN\blfootnote{${}^\ddagger$dyaman@yildiz.edu.tr}${}^\ddagger $ \\
\small{${}^\star$ Research Laboratory of Algebra, Topology, Arithmetic, and Order} \\ \small{ Department of Mathematics, Faculty of Mathematical, Physical and Natural Sciences of
Tunis} \\ \small{Tunis-El Manar University, 2092-El Manar, Tunisia } \\
\small{${}^\dagger$ ${}^\ddagger $  Department of Mathematics} \\ \small{Yildiz Technical University, Istanbul }
}
\date{\today}

\maketitle
\abstract{In this study, we prove among other results that the Dedekind completion of the Riesz tensor product of two order ideals in an Archimedean Riesz space $E$ is again an order ideal in $E$ and the Dedekind completion of the Riesz tensor product of two principal bands in an Archimedean Riesz space $E$ is again a principal band in $E$.}

\parindent=0cm
\parskip=0.5cm

\section{Introduction}
Tensor product of vector lattices and different structures has been a significant field of interest and has been studied and improved in time by authors like H.H.Schaefer, J.J. Grobler, C.C.A. Labuschagne and G. Buskes since D. H. Fremlin first introduced the Riesz tensor product. We mainly refer to \cite{ben2021riesz}, \cite{buskes2016tensor} and \cite{jaber2020fremlin} on the topic.
In \cite{fremlin1972tensor}, Fremlin proved the existence and uniqueness of the tensor product in the following theorem. (For other constructions see \cite{grobler1988tensor} or \cite{schaefer1980aspects}).

\begin{theorem}\label{fremlin}
Let $E$ and $F$ be Archimedean Riesz spaces. Then there is an Archimedean Riesz space $G$ and a Riesz bimorphism $\varphi\colon E\times F\rightarrow G$ such that 
\begin{itemize}
\item[{ (i)}] whenever $H$ is an Archimedean Riesz space and $\psi\colon E\times F\rightarrow H$ is a Riesz bimorphism, there is a unique Riesz homomorphism $T\colon G\rightarrow H$ such that $T\varphi=\psi$;
\item[{(ii)}] $\varphi$ induces an embedding $\hat\varphi\colon E \otimes F\to G$;
\item[{ (iii)} ] (ru-D) $\hat{\varphi}[E \otimes F]$ is dense in $G$ in the sense that for every $w\in G$, there exist $x_0\in E$ and $y_0\in F$ such that for every $\epsilon > 0$, there is an element $v\in \hat{\varphi}[E\otimes F]$ such that $|w-v|\le \epsilon \hat{\varphi}(x_0\otimes y_0)$;
\item[{(iv)}] if $w>0$ in $G$, then there exist $x\in E^+$ and $y\in F^+$ such that $0<\hat{\varphi}(x\otimes y)\le w$.  
\end{itemize}
\end{theorem}

The Archimedean Riesz space $G$ in Theorem \ref{fremlin} is called the \emph{Fremlin tensor product} of $E$ and $F$ and is denoted by $E\bar{\otimes} F$. Any Archimedean Riesz space paired with a Riesz bimorphism satisfying the universal property $(i)$ is Riesz isomorphic to $G$.

The Fremlin tensor product has the following additional properties:
\begin{itemize}
\item[(ru-D$)_+$] (Positive relative uniform density property) For $w\in {E\overline{\otimes}F}_+$, there exists an element $(x,y)\in E_+ \times F_+$ such that for every $\epsilon > 0$, there exists an element $v\in (E_x)_+ \otimes (F_y)_+$ with  $|w-v|\le \epsilon \hat{\varphi}(x\otimes y)$
\item[(PUM)] (Positive universal mapping property) Let $H$ be a relatively uniformly complete Archimedean Riesz space and $\psi : E\times F \to H$ be a positive bilinear mapping. Then there exists a unique positive linear mapping $\tau : E\overline{\otimes}F \to H$ such that $\tau \circ \varphi =\psi$.
\item[(B)] If $h\in E\overline{\otimes}F$,  there exists an element $(x, y)\in E_{+}\times F_{+}$  such that $|h|\leq x\otimes y$.
\end{itemize}

Recently, in \cite{BENAMOR2023127074}, the authors try to give a generalization of the weak mixing in the framework of Riesz spaces. They need the Riesz tensor product of two Dedekind complete Riesz spaces to be Dedekind complete and the Riesz tensor product of two principal bands in a Riesz space to be a principal band in the same Riesz space. Grobler in \cite{grob2022}, showed that the first statement is incorrect by giving a counterexample. Hence, 
he introduced the Dedekind complete tensor product as the Dedekind completion of the Riesz tensor product, that is,  $E\overline{\otimes}_{\delta}F= (E\overline{\otimes}F)^{\delta}$. He stated and proved the following theorem \cite[Theorem 5.2]{grob2022}  which will be useful in the sequel.


\begin{theorem}\label{tensded}
  Let $E,F$ be two Archimedean Riesz spaces. The tensor product $E\overline{\otimes}_{\delta}F$  is a Dedekind complete vector lattice with the following properties:
  \begin{itemize}
      \item[(D3) ]$E\otimes F$  is a vector subspace of $E\overline{\otimes}F$  and $E\overline{\otimes}F$  is a Riesz subspace of $E\overline{\otimes}_{\delta}F$.
      \item[(D5)]If $G$ is a Dedekind complete vector lattice and $\psi$ : $E\times F\rightarrow G$  is an order continuous Riesz bimorphism, then there exists an order continuous Riesz homomorphism $\tau$ : $E\overline{\otimes}_{\delta}F\rightarrow G$ such that $\tau\circ\sigma=\psi$.
      \item[(D4)]The Dedekind complete Riesz subspace generated in $E\overline{\otimes}_{\delta}F$ by $E\otimes F$  is equal to $E\overline{\otimes}_{\delta}F$.
      \item[(OD)] For every $0<h\in E\overline{\otimes}_{\delta}F$ there exists an element $(x, y)\in E_{+}\times F_{+}$ such that $0<x\otimes y\leq h$.
      \item[(B)]  If $h\in E\overline{\otimes}_{\delta}F$,  there exists an element $(x,y)\in E_{+}\times F_{+}$  such that $|h|\leq x\otimes y$.
      \item[(PUM)]Let $G$  be a Dedekind complete Archimedean Riesz space and let $\psi$ : $E\times F\rightarrow G$  be an order continuous positive bilinear mapping. Then there exists a unique order continuous positive linear mapping $\tau$ : $E\overline{\otimes}_{\delta}F\rightarrow G$ such that $\tau\circ\sigma=\psi$.
  \end{itemize}
\end{theorem}

Most recently, G. Buskes and P. Thorn in the preprint \cite{buskes2022} present a counterexample for the tensor product of two ideals of Archimedean Riesz spaces $E$ and $F$ being an ideal in $E\rtp F$. Then, it takes away the hope to obtain that the Riesz tensor product of principal bands is again a principal band. 

In the light of these recent works, we try to present an overall look at the subject, and we show that in the case of ideals and principal bands, it is more convenient to deal with the Dedekind complete tensor product rather than the Fremlin one. We will prove that the Dedekind complete tensor product of two ideals in a Riesz space is again an ideal and the Dedekind complete tensor product of two principal bands is again a principal band in the same Riesz space.

\section{Preliminaries}
Some fundamental notions and tools are given in this section. The reader can find further information with details in \cite{aliprantis2006positive} or in \cite{kusraev2000dominated} or in \cite{MR0511676}.

A real vector space $E$ equipped with an order relation $\ge$ compatible with the algebraic structure of $E$ which has the following properties is called an \emph{ordered vector space}.
\item[i.] If $u\ge v$ for $u,v\in E$ then $u+w \ge v+w$ for all $w\in E$.
\item[ii.] If $u\ge v$ for $u,v\in E$ then $\lambda u \ge \lambda v$ for each $\lambda \ge 0 $.

An ordered vector space $E$ is called a \emph{Riesz space} if the supremum (or infimum) of every nonempty finite subset of $E$ exists in $E$. A Riesz space $E$ is called \emph{Archimedean} if $x,y\in E^+$ and $nx \le y$ for each $n\in \mathbb{N}$ imply $x=0$.
A Riesz space is called \emph{Dedekind complete} whenever every nonempty subset which is bounded above has a supremum (or every nonempty subset which is bounded below has an infimum).
A vector subspace $A$ of an ordered vector space $E$ is said to be \emph{majorizing} in $E$ whenever for each $x\in E$ there exists some $y\in A$ such that $x\le y$.
A Dedekind complete Riesz space $L$ is called the \emph{Dedekind completion} of the Riesz space $E$ whenever $E$ is Riesz isomorphic to a majorizing order dense Riesz subspace of $L$. Every Archimedean Riesz space has a unique (up to a lattice isomorphism) Dedekind completion \cite[Theorem 2.24]{aliprantis2006positive}.

Recall that a \emph{Riesz subspace} of a Riesz space $E$ is a subspace of $E$ that is closed under the lattice operations in $E$.
A subset $A$ of a Riesz space is called \emph{solid} whenever $|x|\leq |y|$ and $y\in A$ implies $x\in A$. A solid vector subspace of a Riesz space is called an \emph{ideal}. It is easy to see that every ideal is a Riesz subspace.
The ideal generated by a vector $x$ is defined as
\[
E_x=\{y\in E\ :   \text{ there exist } \lambda >0 \text{  such that  } |y|\le \lambda |x| \}
\]
and is called a \emph{principal ideal}.

 
A Riesz subspace $A$ of a Riesz space $E$ is called \emph{order dense} in $E$ whenever for every $0<x \in E$ there exists some $y\in A$ with $0< y\le x$.

A net $\{x_\alpha\}$ in a Riesz space is said to be \emph{order convergent} to a vector $x$ (denoted by $x_\alpha \xrightarrow{o} x$), whenever there exist $\alpha_0$ and another net $\{y_\beta\}$ which satisfies $y_\beta \downarrow 0$ and $|x_\alpha - x|\le y_\beta$ for all $\alpha \geq \alpha_0$. A subset $A$ of an Archimedean Riesz space $E$ is called \emph{order closed} whenever any convergent net $\{x_\alpha\}$ in $A$ order converges to a vector $x$ in $A$.
An order closed ideal is called a \emph{band}. The band generated by a vector $x$ in a Riesz space $E$ is called a \emph{principal band} and given by 
\[
B_x=\{y\in E : |y|\wedge n|x|\uparrow |y|\}.
\]

The disjoint complement $A^d$ of a nonempty subset $A$ of a Riesz space $E$ is defined by 

\[ A^d = \{u\in E : u\wedge v =0 \text{ for all } v\in A \}.
\]
 Note that $A\subseteq A^{dd}$ where $A^{dd}=(A^d)^d$. In an Archimedean Riesz space, $A^{dd}$ is the band generated by $A$ and if $A$ is a band  then $A= A^{dd}$.

A band $B$ in a Riesz space $E$ that satisfies $E=B\oplus B^d$ is called a \emph{projection band}.
An operator $P: V\to V$ on a vector space $V$ is called a \emph{projection} if $P^2=P$. A projection $P$ defined on a Riesz space which is at the same time a positive operator is called a \emph{positive projection}. Let $B$ be a projection band in a Riesz space $E$. Then every $x\in E$ has a unique decomposition $x=x_1 + x_2$ where $x_1 \in B$ and $x_2\in B^d$. In this case the projection $P_B: E\to E$ defined with 
\[ P_B(x)=x_1 \] is called a \emph{band projection}.

A vector $e>0$ in a Riesz space $E$ is called a \emph{weak order unit} whenever the band generated by $e$ is equal to $E$, i.e., $B_e=E$. Every positive vector of a Riesz space is a weak order unit in the band it generates.

Let $e$ be a positive vector of a Riesz space $E$. A positive vector $x$ in $E$ is called the \emph{component} of $e$ whenever $x\wedge (e-x)=0$.

A Riesz subspace $A$ of a Riesz space $E$ is called \emph{regular} if the embedding of $A$ into $E$ preserves the arbitrary suprema and infima.

Next, we give some fundamental results that are of significance for our work. See \cite{aliprantis2003locally}, \cite{aliprantis2006positive} and \cite{zaanen2012introduction} for the proofs.

 \begin{theorem}\label{regular} \cite[Theorem 1.23]{aliprantis2003locally} 
 Every order dense Riesz subspace of a Riesz space is a regular Riesz subspace.
 \end{theorem}

 \begin{theorem}\cite[Theorem 1.24]{aliprantis2003locally}
 Every ideal $I$ of a Riesz space $E$ is order dense in $I^{dd}$. In particular, an ideal $I$ is order dense in $E$ if and only if $I^d = \{0\}$.
 \end{theorem}
\begin{theorem} \cite[Theorem 1.46]{aliprantis2003locally}
Let $E$ be a Riesz space. Then the following assertions hold.
\begin{enumerate}
    \item A band $B$ of $E$ is a projection band if and only if for each $u\in E^+$ the supremum $ \sup \{ v\in B:  0\le u\le v \}$ exists in $E$. Moreover, if $B$ is a projection band then for each $u\in E^+$ 
    $$P_B(u)=\{ v\in B:  0\le u\le v \}.$$
    \item A principal band $B_u$ of $E$ is a projection band if and only if for each $v\in E^+$ the supremum
    $\sup \{ v\wedge n|u| :  n\in \mathbb{N}\}$ exists in $E$. Moreover, if $B_u$ is a projection band then for each $v\in E^+$ 
    $$P_u(v)=\{ v\wedge n|u|:  n\in \mathbb{N} \}.$$
\end{enumerate}
\end{theorem}

\begin{theorem}\cite[Theorem 12.2]{zaanen2012introduction}
Let $E$ be a Dedekind complete Riesz space. Then every band $B$ in $E$ is a projection band.
\end{theorem}

Throughout this paper all Riesz spaces will be assumed to be Archimedean.

\section{Order Continuity of the Riesz Tensor Product}

In this section we prove that the embedding of $E \times F$ in $E \bar{\otimes} F$ is order continuous. Our alternative proof will precise the proof in \cite[3.4]{grob2022}. Yoshida proved in his early work \cite{yosida194270} the following representation theorem which will be useful for our work. (For further details about the representation theorem see \cite{groenewegen2016spaces}, \cite{MR0511676} and \cite{yosida194270}).

\begin{theorem}\label{yos}
Let $E_u$ be a Dedekind complete Riesz space with a strong unit $u$. Then there are a compact topological space $X$ and a Riesz isomorphism $\varphi:\ E_u \to C(X)
$ such that\[\begin{array}{lcl}
\varphi(u)&=&\mathds{1}
\end{array}
\]
\end{theorem}

\begin{lemma}\label{Yoshida}
Let $E$ and $F$ be Archimedean Riesz spaces, $x$ and $x'$ be strictly positive elements in $E$ and $y, y'$ be strictly positive elements in $F$. If $x \otimes y \leq x' \otimes y'$, then there exist two strictly positive real numbers $\alpha$ and $\beta$ such that  $x \leq \alpha x'$ and $y \leq \beta y'$.
\end{lemma}

\begin{proof}
Let $E_{x \vee x'}$ and $F_{y \vee y'}$ be the principal ideals generated by $x \vee x'$ and $y \vee y'$, respectively. As in the construction of the Fremlin tensor product made by Schaefer in  \cite{schaefer1980aspects}, the tensor product $E_{x \vee x'} \otimes F_{y \vee y'}$ becomes a vector subspace of $C(K_{x \vee x'}) \otimes C(K_{y \vee y'})$ which in turn can be viewed as a subspace of $C(K_{x \vee x'} \times K_{y \vee y'})$, the space of all functions of the form $h(s,t) = \sum_i f_i(s)g_i(t)$ where $f_i\in C(K_{x \vee x'})$, $g_i\in  C(K_{y \vee y'})$ and $K_{x \vee x'}$ and $ K_{y \vee y'}$ are compact topological spaces. $x$ and $x'$ can be viewed as $\hat{x}$ and $\hat{x}'$ in  $C(K_{x \vee x'})$ and $y$ and $y'$ can be viewed as $\hat{y}$ and $\hat{y}'$ in  $C(K_{y \vee y'})$. Since $x \otimes y \leq x' \otimes y'$, it follows that $$ \hat{x}(s)\hat{y}(t)\leq \hat{x}'(s)\hat{y}'(t)$$ for every $s$ and $t$ in $K_{x \vee x'}$ and $ K_{y \vee y'}$ respectively. Since $y$ is strictly positive, there exists some $t_0$ in  $K_{y \vee y'}$ such that $\hat{y}(t_0) > 0$. It follows that $$ \hat{x}(s)\leq \frac{\hat{y}'(t_0)}{\hat{y}(t_0)}\hat{x}'(s)$$ for every $s$ in $K_{x \vee x'}$. Fix $$ \alpha = \frac{\hat{y}'(t_0)}{\hat{y}(t_0)}.$$  Then $x \leq \alpha x'$. We proceed in the same way to show that  $y \leq \beta y'$ and this completes the proof.
\end{proof}

Now, we gathered all the ingredients we need in order to prove the main result of this section.

\begin{theorem}\label{oc}
    Let $E$ and $F$ be two Archimedean Riesz spaces and let $\sigma$ be a bilinear map defined by:
    \[\begin{array}{lcll}
        \sigma: & E \times F & \longrightarrow & E\rtp F  \\
         &  (x,y) & \longmapsto& x \otimes y. 
    \end{array}\]
    Then $\sigma$ is order continuous.
\end{theorem}

\begin{proof}
The result will be proved in several steps.
\begin{itemize}
    \item     Let $(x_\alpha)$ be a net in $E$ such that $x_\alpha \downarrow 0$ and let $y$ be a strictly positive element in $F$. Then $x_\alpha \otimes y \downarrow 0$. \\
    In order to prove this assertion, note that $x_\alpha \otimes y$ is a decreasing net and minimized by $0$. So, there exists some positive element $u$ in $\overline{E\rtp F}^\delta $ such that $x_\alpha \otimes y \downarrow u$. Assume that $u> 0$. The order denseness of $E\rtp F $ in $\overline{E\rtp F}^\delta $ together with (iv) in Theorem \ref{fremlin}, yield the existence of $0< x'$ in $E$ and $0< y'$ in $F$ such that $$0\leq x'\otimes y' \leq u \leq x_\alpha \otimes y \text{ for every } \alpha.$$ From Lemma \ref{Yoshida}, it follows that there exists some $\beta > 0$ such that $$x' \leq \beta x_\alpha  \text{ for every } \alpha,$$ and then $$\frac{1}{\beta} x' \leq x_\alpha  \text{ for every } \alpha.$$ The latter fact yields $x'\leq 0$ which is a contradiction and then $u=0$. It follows that $x_\alpha \otimes y \downarrow 0$ in $\overline{E\rtp F}^\delta $ and in $E\rtp F $.
    \item  Let $(x_\alpha)$ be a net in $E$ such that $x_\alpha \downarrow 0$ and $(y_\beta)$  be a net in $F$ such that $y_\beta \downarrow 0$. Then the double net $x_\alpha \otimes y_\beta \downarrow 0$. \\
    As $(y_\beta)$ is decreasing to zero, $(y_\beta)$ has a bounded tail. Then there exists some $y'$ in $F_+$ such that $y_\beta \leq y'$. It follows that $$0 \leq x_\alpha \otimes y_\beta  \leq x_\alpha \otimes y'.$$ This completes the proof.
\end{itemize}

\end{proof}

\section{The Riesz tensor product of Riesz subspaces}

In this section, principal ideals and principal bands of Riesz spaces will be our focus of interest. 

\begin{proposition}
\label{density}
 Let $E$ and $F$ be Riesz spaces. Let $e$ be a positive element in $E$, and $f$ be a positive element in $F$. Let $E_e$ and $E_f$ be the principal ideals generated by $e$ and $f$, respectively. Then $E_e \rtp F_f$ is an order dense Riesz subspace of $(E \rtp F)_{e \otimes f}$, where $(E \rtp F)_{e \otimes f}$ is the principal ideal in  $(E \rtp F)$ generated by ${e \otimes f}$.
\end{proposition}
\begin{proof}
By \cite[Proposition 3.1]{grob2022} and the boundedness Property (B), it follows that $E_e \rtp F_f$ is a Riesz subspace of $(E \rtp F)_{e \otimes f}$.
Let $0 < u \in (E \rtp F)_{e \otimes f}$. Then there exists a positive real number $\alpha$ such that $0 < u \leq \alpha e \otimes f$. By the Property (OD), there exist positive elements $x\in E$ and $y\in F$
 such that $ 0 < x \otimes y \leq u$. It follows that $0 < x \otimes y \leq \alpha e \rtp f$. By the Lemma \ref{Yoshida}, we have $0 \leq x \in E_e$ and $0 \leq y \in  F_f$ and hence $ x \otimes y \in E_e \rtp F_f$. \end{proof}

 Now we have all the material in order to prove the next theorem. 
 
 \begin{theorem}\label{princbands}
    Let $E$ and $F$ be Riesz spaces. Let $e$ be a positive element in $E$, and $f$ be a positive element in $F$. Let $B_e$ and $B_f$ be the principal bands generated by $e$ and $f$, respectively. Then $B_e \rtp B_f$ is an order dense Riesz subspace in the principal band in $E \rtp F$ generated by $ e \otimes f$. 
 \end{theorem}

 \begin{proof}
 First, we need to show that $B_e \rtp B_f \subset  \{e \otimes f \}^{dd}$. To this end, pick a positive element $u$ in $E \rtp F$ such that $u \wedge e\otimes f = 0$, and a positive element $v$ in $B_e \rtp B_f$. By Property (B), there exist positive elements $x\in B_e$ and $y\in B_f$ such that $$0 \leq v \leq x\otimes y.$$ Since $B_e$ and $B_f$  are principal bands, it follows that \begin{align*}
     x &= \bigvee_n x\wedge n e \\
     y &= \bigvee_m y \wedge m f
 \end{align*} 
From the order continuity of the Riesz tensor product (see  \cite[Corollary 3.4]{grob2022}), it follows that \begin{align*}
    0 &\leq u \wedge v \leq u \wedge x\otimes y \\
    0 &\leq u \wedge v \leq u\wedge \bigvee_n \bigvee_m (x\wedge ne)\otimes (y\wedge mf)\\
    0 &\leq u \wedge v \leq \bigvee_n \bigvee_m (u \wedge (x\wedge ne)\otimes (y\wedge mf))\\
    0 &\leq u \wedge v \leq \bigvee_n \bigvee_m (u \wedge  ne\otimes  mf)\\
    0 &\leq u \wedge v\leq \bigvee_n \bigvee_m (n + m) (u \wedge  e\otimes  f) = 0.\\
\end{align*}
Therefore $B_e \rtp B_f \subset  \{e \otimes f \}^{dd}$.
Now, we have that $$E_e \rtp F_f \subset B_e \rtp B_f \subset  \{e \otimes f \}^{dd}$$ and $$E_e \rtp F_f \subset (E \rtp F)_{e \otimes f}\subset  \{e \otimes f \}^{dd}.$$
 
  By \cite[Theorem 1.24]{aliprantis2003locally}, $(E \rtp F)_{e \otimes f}$ is order dense in $ \{e \otimes f \}^{dd}$.  It follows from Proposition \ref{density},  that $B_e \rtp B_f$ is order dense in $\{e \otimes f\}^{dd}$. 
\end{proof}

\begin{corollary}
    Let $E$ and $F$ be Riesz spaces, $e$ and $f$ be positive elements in $E$ and $F$, respectively. Let $B_e$ and $B_f$ be the principal bands generated by $e$ and $f$, respectively. Then $B_e \rtp B_f$ is order dense in the principal band generated by $e\otimes f$ in $E \rtpd F$.
\end{corollary}

    \begin{proof}
        We will denote by $(e\otimes f)^{dd}$ the principal band generated by $e\otimes f$ in $E\rtp F$ and by $(e{\otimes}_\delta f)^{dd}$ the principal band generated by $e\otimes f$ in $E\rtpd F$. Now take some strictly positive element $u\in (e{\otimes}_\delta f)^{dd}$. As $E\rtp F$ is order dense in $E\rtpd F$, there exists some $w$ such that $0<w\le u$. Pick $v\in E\rtp F$ such that $v\wedge e\otimes f=0$ in $E\rtp F$. Since $E\rtp F$ is regular in $E\rtpd F$ (see \ref{regular}), $v\wedge e\otimes f=0$ in $E\rtpd F$ as well. So we have that $v\wedge u =0$ and then $0\le w\wedge v \le u\wedge v =0$. This implies that $w$ belongs to $(e{\otimes} f)^{dd}$. Therefore $(e{\otimes} f)^{dd}$ is order dense in $(e{\otimes}_\delta f)^{dd}$. Together with $B_e \rtp B_f$ being order dense in $(e{\otimes} f)^{dd}$, this shows that $B_e \rtp B_f$ is order dense in $(e{\otimes}_\delta f)^{dd}$.
    \end{proof}

 The next proposition is a generalization of  \cite[Porposition 7.1]{grob2022}.
 \begin{proposition}\label{weak}
 Let $e$ and $f$ be weak order units of the Riesz spaces $E$ and $F$, respectively. Then $e\otimes f$ is a weak unit in $E \rtp F$.
 \end{proposition}
 
 \begin{proof}
 Let $0\leq u$ in $E \bar{\otimes}F$ such that $u\wedge e \otimes f=0$, we have to show that $u=0$.
\begin{itemize}
\item First assume that $u=x\otimes y$ in $E_{+} \otimes F_{+}$ then from $x\otimes y \wedge  e \otimes f = 0$ it follows that \[(x\wedge n e) \otimes (y \wedge f)= 0 \ \ \forall n \in \mathbb{N}.\] The fact that $\displaystyle  \ \{ x \wedge n e \} \uparrow x $, together with the order continuity of the Riesz tensor product (see \cite{grob2022}), yield to  \[x \otimes (y \wedge f)= 0.\]We proceed in the same way to obtain that $ u = x \otimes y =0.$
\item Now, if $0 < u \in E \bar{\otimes}F$ and $u \wedge e \otimes f=0$, by Property (OD) there exist $x\in E_+$ and $y\in F_+$ such that \[0<x \otimes y \leq u.\] But then $x\otimes y \wedge  e \otimes f=0$ and this contradicts with $x \otimes y =0$ , therefore $u$ has to be null. 
\end{itemize}
With the last point of the proof we obtain that $e \otimes f$ is a weak order unit in $E \bar{\otimes}F$.
 \end{proof}

\section{The Dedekind complete Riesz tensor product of Riesz subspaces}

In this section, we give the results on Dedekind complete tensor product of ideals and projection bands.

\begin{lemma}\label{idealded}
 Let $E$ and $F$ be Riesz spaces. Let $e$ be a positive element in $E$, and $f$ be a positive element in $F$. Let $E_e$ be the principal ideal generated by $e$ and  $E_f$ be the principal ideal generated by $f$, then the principal ideal $(E \rtp_\delta F)_{e\otimes f}$ is equal to $\overline{(E \rtp F)_{e \otimes f}}^\delta$, where  $\overline{(E \rtp F)_{e \otimes f}}^\delta$ is the Dedekind completion of $(E \rtp F)_{e \otimes f}$.
\end{lemma}

\begin{proof}
First observe that $(E \rtpd F)_{e \otimes f}$ is a Dedekind complete ideal containing $(E \rtp F)_{e \otimes f}$ (see for example \cite[Theorem 12.4]{zaanen2012introduction}). It follows that $\overline{(E \rtp F)_{e \otimes f}}^\delta$ is included in $(E \rtpd F)_{e \otimes f}$.
Now take a positive element $u$ in $(E \rtpd F)_{e \otimes f}$. Then there exists some positive real $\alpha$ such that $u \leq \alpha (e \otimes f)$. The fact that $u$ is in $(E \rtpd F)$ is equivalent  to (see for example \cite{aliprantis1984order}.) 
 \begin{align*} 
u &=   \sup \{v \in E \rtp F \text{ such that } 0 \leq v \leq u \}\\
 &=  \sup \{v \in E \rtp F \text{ such that } 0 \leq v \leq u \leq \alpha (e \otimes f) \} \\
  & = \sup \{v \in (E \rtp F)_{e \otimes f} \text{ such that } 0 \leq v \leq u  \}.
\end{align*}
We obtain that $u$ is in $\overline{(E \rtp F)_{e \otimes f}}^\delta$ which makes an end to our proof.
\end{proof}

 We give next our result for Dedekind complete tensor product of principal ideals and then later for ideals.
 
\begin{theorem}\label{idealprincded}
 Let $E$ and $F$ be Riesz spaces. Let $e$ be a positive element in $E$, and $f$ be a positive element in $F$. Let $E_e$ be the principal ideal generated by $e$ and  $E_f$ be the principal ideal generated by $f$, then $E_e \rtp_\delta F_f$ is equal to the principal order ideal $(E \rtp_\delta F)_{e \otimes f}$.
\end{theorem}

\begin{proof}
From Lemma \ref{density}, it follows that $E_e \rtp F_f$ is an order dense Riesz subspace of $(E \rtp F)_{e \otimes f}$. Since $(E \rtp F)_{e \otimes f}$ is full in $(E \rtpd F)_{e \otimes f}$ (see \cite{aliprantis1984order} for example), it follows that $E_e \rtpd F_f$ is an order dense Riesz subspace of $(E \rtpd F)_{e \otimes f}$. Now Pick $u$ a positive element in $(E \rtpd F)_{e \otimes f}$, there exists some positive real number $\alpha$ such that $0 \leq  u \leq \alpha e\otimes f$. This, together with Lemma \ref{density}, yield to
\begin{align*} 
u &=   \sup_{(E \rtpd F)} \{v \in E_e \rtpd F_f \text{ such that } 0 \leq v \leq u \leq \alpha e \otimes f \}
\end{align*}
The set $\mathfrak{D} =     \{v \in E_e \rtpd F_f \text{ such that } 0 \leq v \leq u \leq \alpha e \otimes f \}$ is bounded above in the Dedekind complete Riesz space $E_e \rtpd F_f$, then there is $u^\star$ in $E_e \rtpd F_f$ such that \begin{align*} 
u^\star &=   \sup_{E_e \rtpd F_f} \{v \in E_e \rtpd F_f \text{ such that } 0 \leq v \leq u \leq \alpha e \otimes f \}
\end{align*}
  \cite[Lemma 2.5]{gao2017uo} and \cite[Theorem 1.23]{aliprantis2003locally} together with the order denseness of $E_e \rtpd F_f$ in $(E \rtpd F)_{e \otimes f}$, yield to $$u = u^\star.$$ That is $E_e \rtpd F_f$ = $(E \rtpd F)_{e \otimes f}$, which makes an end to the proof.
\end{proof}

\begin{corollary}
Let $E$ and $F$ be two Archimedean Riesz spaces. Let $e$ be a positive element in $E$, and $f$ be a positive element in $F$. Let $E_e$ be the principal ideal generated by $e$ and  $F_f$ be the principal ideal generated by $f$. If $E\rtp F$ is Dedekind complete then $E_e \rtpd F_f$ is an ideal in $E\rtp F$.
\end{corollary}

The next corollary can be easily obtained when we consider $E\rtpd \mathbb{R}$.

\begin{corollary}
Let $E$ be a Riesz space and $0 \leq e \in E$. Then the Dedekind completion of the principal ideal generated by $e$ is the principal ideal generated by $e$ in the Dedekind completion. That is, $\overline{E_e}^\delta =  (\overline{E}^\delta)_e$
\end{corollary}

\begin{proof}
Let $E$ be a Riesz space and $0 \leq e \in E$.  Theorem \ref{idealprincded} yields to, $$\begin{array}{lcl}
 \overline{E_e}^\delta    &  = & E_e\rtpd \mathbb{R} \\
     & = & (E\rtpd \mathbb{R})_{e \otimes 1}\\
     & = &  (\overline{E}^\delta)_e
\end{array} $$
which makes an end to our proof.
\end{proof}

\begin{theorem}\label{ideals}
 Let $E$ and $F$ be two Archimedean Riesz spaces and $A$ and $B$ be two ideals of $E$ and $F$ respectively. Then $A \rtpd B$ is an ideal in $E\rtpd F$.
\end{theorem}
 
\begin{proof}
Pick a positive element $u$ in $A \rtpd B$ and a positive element $v$ in $E\rtpd F$, such that $$0 \leq v \leq u.$$
From the Property (B), there are $(x,y)$ in $A_+ \times B_+$ such that $$ 0 \leq v \leq u \leq x\otimes v.$$ The principal ideals generated by $x$ in $A$ and $y$ in $B$ are principal ideals in $E$ and $F$ respectively, it follows from Theorem \ref{idealprincded} that $u$ is in the principal ideal $E_x \rtpd F_y$. This yields to $$v \in E_x \rtpd F_y \subset A\rtpd B.$$ 
\end{proof}

As an immediate consequence, one can derive the next corollary.
\begin{corollary}
Let $E$ and $F$ two Archimedean Riesz spaces and $A$ and $B$ be two ideals of $E$ and $F$ respectively. If $E\rtp F$ is Dedekind complete then $A \rtpd B$ is an ideal in $E\rtp F$.
\end{corollary}

\begin{lemma}
    Let $E$ and $F$ be Riesz spaces, $e$ and $f$ be positive elements in $E$ and $F$, respectively. Let $B_e$ be the principal band generated by $e$, $B_f$ be the principal band generated by $f$ and $v$ be an element of $\overline{(B_e \rtp B_f)}^\delta$. Then 
    \[ v= \sup \{u\in B_e\otimes B_f: 0\leq u\leq v\}.\] 
    \begin{proof}
    The set $\{u\in B_e\otimes B_f: 0\leq u\leq v\}$ is bounded above by $v$ and the supremum exists. Let \[w= \sup \{u\in B_e\otimes B_f: 0\leq u\leq v\}\]
    and let $w\ne v$. Then $0\le w < v$ and it implies that $w\in \overline{(B_e \rtp B_f)}^\delta$. It follows that $0\le v-w$. By the property (OD) in \cite{grob2022}, there exist $(x,y)\in {B_e}^+ \times {B_f}^+$ such that 
    \begin{align*}
        0 &< x\otimes y <v-w \\
        0 &< w< w+x\otimes y< v \\
        u & +x\otimes y \le w \\
        u &\le w-x\otimes y
       \end{align*} for all $u\in B_e\otimes B_f$.
       It follows that $w\le w-(x\otimes y)$ which is impossible and we have a contradiction. So $w=v$ and this completes the proof.

    \end{proof}
\end{lemma}

 Now we have the material which we need to prove the next theorem.
 \begin{theorem}\label{pricband}
Let $E$ and $F$ be Riesz spaces. Let $e$ be a positive element in $E$, and $f$ be a positive element in $F$. Let $B_e$ be the principal band generated by $e$ and  $B_f$ be the principal band generated by $f$. Then $B_e \rtp_\delta B_f$ is equal to the principal band in $E\rtpd F$ generated by  ${(e \otimes f)}$.
 \end{theorem}

\begin{proof}

$B_e \rtpd B_f$ is an order dense ideal in $(e\rtp f)^{dd}$ (see \ref{princbands}). It remains to prove that $(e\rtp f)^{dd}$ is included in $B_e \rtpd B_f$ . To this aim, let
\[\begin{array}{lcll}
        \sigma: & E \times F & \longrightarrow & B_e\rtpd B_f  \\
         &  (x,y) & \longmapsto& P_e(x)\otimes P_f(y) 
    \end{array}\]

be an order continuous Riesz bimorphism. By \cite[Theorem 5.1]{grob2022}, there exists an order continuous Riesz homomorphism $T$ such that
\[\begin{array}{lcll}
        T: & E\rtpd F & \longrightarrow & B_e\rtpd B_f  \\  
    \end{array}\]

and        $  T(x\otimes y) = \sigma (x,y) = P_e(x)\otimes P_f(y) $ for all $x$ in $E$, $y$ in $F$.

Now pick a positive element $u\in (e\rtp f)^{dd}$. Since $B_e\rtp B_f$ is order dense in $(e\rtp f)^{dd}$, we can find a positive net $u_\alpha \in B_e \rtp B_f$ such that $u_\alpha \xrightarrow{o} u$. It follows from the order continuity of $T$ that $T(u_\alpha) \xrightarrow{o} T(u)$.
As $u_\alpha \in B_e \rtp B_f$, we can write
\[
u_\alpha =\sum_{i=1}^{n_1} \bigvee_{j=1}^{n_2} \bigwedge_{k=1}^{n_3} {x^\alpha}_{ijk} \otimes {y^\alpha}_{ijk}
\]
for all $i,j,k,\alpha$, where ${x^\alpha}_{ijk} \in B_e$  and ${y^\alpha}_{ijk} \in B_f$.
Since $T$ is a Riesz homomorphism, it follows that
\begin{align*}
    T(u_\alpha) &= \sum_{i=1}^{n_1} \bigvee_{j=1}^{n_2} \bigwedge_{k=1}^{n_3} T({x^\alpha}_{ijk} \otimes {y^\alpha}_{ijk}) \\
    &= \sum_{i=1}^{n_1} \bigvee_{j=1}^{n_2} \bigwedge_{k=1}^{n_3} P_e({x^\alpha}_{ijk}) \otimes P_f({y^\alpha}_{ijk}) \\
    &= \sum_{i=1}^{n_1} \bigvee_{j=1}^{n_2} \bigwedge_{k=1}^{n_3} {x^\alpha}_{ijk} \otimes {y^\alpha}_{ijk} \\
    &= u_\alpha
    \end{align*}
    Then we have $T(u_\alpha) \xrightarrow{o} T(u) =u$ which implies $u\in B_e\rtpd B_f$. This completes the proof.
\end{proof}

The next corollary follows easily from Theorem \ref{pricband}, when we consider $E\rtpd \mathbb{R}$.

\begin{corollary}
Let $E$ be a Riesz space and $0 \leq e \in E$. Then the Dedekind completion of the principal band generated by $e$ is the principal band generated by $e$ in the Dedekind completion. That is, $\overline{B_e}^\delta =  (\overline{B}^\delta)_e$.
\end{corollary}

\begin{proof}
Let $E$ be a Riesz space and $0 \leq e \in E$.  Theorem \ref{pricband} yields to, $$\begin{array}{lcl}
 \overline{B_e}^\delta    &  = & B_e\rtpd \mathbb{R} \\
     & = & ({e \otimes 1})^{\text{dd}}\\
     & = &  (\overline{B}^\delta)_e
\end{array} $$
and this completes the proof.
\end{proof}

Now we have the results on the Dedekind complete tensor product of projection bands.

\begin{corollary}
Let $E$ and $F$ be two Riesz spaces with weak units $e$ and $f$ respectively. Then if $A_1$ and $A_2$ are projection bands in $E$ and $F$ respectively then $A_1\rtpd A_2$ is a projection band in $E\rtpd F$.
\end{corollary}

\begin{proof}
If $A_1$ is a projection band in $E$, then $A_1 = B_{P_{A_1}(e)}$, the principal band in $E$ generated by $P_{A_1}(e)$. $A_2$ is then the principal band in $F$ generated by $P_{A_2}(f)$. It follows that $$A_1 \rtpd A_2 = B_{P_{A_1}(e)} \rtpd B_{P_{A_2}(f)}.$$
This together with Theorem \ref{pricband} completes our proof.
\end{proof}

As every band in a Dedekind complete Riesz space is a projection band (see \cite[Theorem 12.2]{zaanen2012introduction}), the next corollary follows immediately.

\begin{corollary}
Let $E$ and $F$ be two Dedekind complete Riesz spaces with weak units $e$ and $f$ respectively. If $A_1$ and $A_2$ are  bands in $E$ and $F$, respectively, then $A_1\rtpd A_2$ is a projection band in $E\rtpd F$.
\end{corollary}

\bibliographystyle{plain}
\bibliography{biblio}
\end{document}